\documentclass[12pt]{amsart}
\usepackage[alphabetic]{amsrefs}
\usepackage{mathdots, blkarray}
\usepackage{multicol}
\usepackage{graphicx}
\usepackage{amscd}
\usepackage{upgreek}
\usepackage{stmaryrd}
\usepackage{longtable}
\usepackage[T1]{fontenc}
\usepackage{latexsym, amsmath, amssymb, amsthm}
\usepackage[svgnames]{xcolor}
\usepackage{rsfso}
\usepackage[mathscr]{eucal}
\usepackage{mathtools}
\usepackage{mathptmx}
\usepackage{titletoc}
\usepackage{wrapfig}
\usepackage{float}
\usepackage{xypic}
\usepackage{microtype}
\usepackage{dsfont}
\usepackage{xcolor}
\usepackage{color}
\usepackage[colorlinks = true,
            linkcolor  = DarkBlue,
            urlcolor   = DarkRed,
            citecolor  = DarkGreen]{hyperref}
\usepackage{hyperref}
\allowdisplaybreaks	
\usepackage[centering, includeheadfoot, hmargin=1.0in, tmargin=0.5in, 
  bmargin=0.6in, headheight=6pt]{geometry}
\newtheorem{theorem}{Theorem}[section]
\newtheorem{prop}[theorem]{Proposition}
\newtheorem{defn}[theorem]{\rm\textsc{Definition}}
\newtheorem{lem}[theorem]{Lemma}
\newtheorem{coro}[theorem]{Corollary}

\newtheorem{thm}[theorem]{Theorem}

\newtheorem{rem}[theorem]{\rm\textsc{Remark}}
\newtheorem{exam}[theorem]{\rm\textsc{Example}}
\DeclareMathOperator{\Hom}{Hom}

\newcommand{\B}{\mathcal{B}} 
 
\newcommand{\PP}{\mathcal{P}} 
\newcommand{\NN}{\mathcal{N}} 
\newcommand{\C}{\mathbb{C}}

\newcommand{\N}{\mathbb{N}}

\newcommand{\RA}{\Longrightarrow} 
 
\newcommand{\ra}{\longrightarrow}

\newcommand{\ha}{\widehat}

\newcommand{\hbo}{$\hfill\Diamond$}

\begin{document}
\title{Deformations of left-symmetric color algebras} 
\def\shorttitle{Deformations of left-symmetric color algebras}

\author{Yin Chen}
\address{School of Mathematics and Physics, Key Laboratory of ECOFM of 
Jiangxi Education Institute, Jinggangshan University,
Ji'an 343009, Jiangxi, China \& Department of Finance and Management Science, University of Saskatchewan, Saskatoon, SK, Canada, S7N 5A7}
\email{yin.chen@usask.ca}

\author{Runxuan Zhang}
\address{Department of Mathematics and Information Technology, Concordia University of Edmonton, Edmonton, AB, Canada, T5B 4E4}
\email{runxuan.zhang@concordia.ab.ca}

\begin{abstract}
We develop a deformation theory for finite-dimensional left-symmetric color algebras, which can be used to construct new algebraic structures and interpret left-symmetric color cohomology spaces of lower degrees. We explore equivalence classes and extendability of deformations for a fixed left-symmetric color algebra, demonstrating that each infinitesimal deformation is nontrivially extendable if the third cohomology subspace of degree zero is trivial. We also study Nijenhuis operators and  Rota-Baxter operators on a left-symmetric color algebra, providing a better understanding of the equivalence class of the trivial infinitesimal deformation.
\end{abstract}

\date{\today}
%\subjclass[2010]{13A50.}
\maketitle \baselineskip=16.5pt

%%%%%%%%%%%%%%%%%%%%%%%%%%%Contents%%%%%%%%%%%%%%%%%%%%%%%%
%\textcolor{blue}{\tableofcontents{}}
\dottedcontents{section}[1.16cm]{}{1.8em}{5pt}
\dottedcontents{subsection}[2.00cm]{}{2.7em}{5pt}
%\dottedcontents{subsubsection}[2.86cm]{}{3.4em}{5pt}

%%%%%%%%%%%%%%%%%%%%%%%%%%%Sections%%%%%%%%%%%%%%%%%%%%%%%%
\section{Introduction}
\setcounter{equation}{0}
\renewcommand{\theequation}
{1.\arabic{equation}}
\setcounter{theorem}{0}
\renewcommand{\thetheorem}
{1.\arabic{theorem}}

\noindent The present article is a continuation of the work in \cite{CZ25} where we developed a new cohomology theory for left-symmetric color algebras. The deformation theory of algebraic structures has been a classical field with a long history that dates back to  Gerstenhaber's work \cite{Ger64} (for associative rings and algebras) and Nijenhuis-Richardson \cite{NR64}
(for Lie algebras but focusing on connections between deformation and cohomology). The deformation theory of nonassociative algebras has become an active topic because it serves not only as an indispensable tool for constructing new algebraic structures, but also has substantial ramifications in the understanding of the corresponding cohomology theory. 
The primary objective of this article is to develop a deformation theory for finite-dimensional left-symmetric color algebras
and apply it to interpret left-symmetric color cohomology spaces. 

Connecting with various branches of mathematics and mathematical physics such as left-invariant affine structures on Lie groups, rooted tree algebras, and vertex algebras in conformal field theory, left-symmetric (or pre-Lie) algebras have been a significant and popular subject in the theory of nonassociative algebras; see for example \cite{Bai20} for an introductory reference. 
Left-symmetric color algebras form a family of nonassociative and Lie admissible graded algebras, generalizing left-symmetric algebras and
left-symmetric superalgebras (see \cite{ZB12,ZHB11}) and also having a close relationship with Lie color algebras.   
While the theory of Lie color algebras has been relatively well understood, left-symmetric color algebras are in their early stages with only a few results known; see  \cite{CCD17},  \cite{NBN09}, and \cite{CZ25}.

In our previous work \cite{CZ25}, we systematically developed a cohomology theory for a finite-dimensional 
left-symmetric color algebra $A$ (graded by an abelian group $G$ and with respect to a skew-symmetric bicharacter $\upepsilon$ of $G$) and its bimodule $V$ over a field. In particular, we proved that the cohomology space $H^{n+1}(A,V)$
can be computed by the $n$-th Lie cohomology of $[A]$ in coefficients $\Hom(A,V)$, where 
$[A]$ denotes the corresponding Lie color algebra; see \cite[Theorem 1.1]{CZ25} for details. 
This crucial result makes computing left-symmetric color cohomology feasible.
However, to better understand the theory of left-symmetric color cohomology, we need to 
develop a suitable deformation theory to interpret the left-symmetric color cohomology spaces of lower degrees. 

Roughly speaking, a formal deformation of a left-symmetric color algebra $A$ is a polynomial $F_\uplambda(x,y)$ in the variable $\uplambda$ for which the coefficients are taken from the second cochain subspace of degree zero, with the standard bimodule $A$.
The first important role the deformations of $A$ play is that it can be used to construct new left-symmetric color algebras on
the same underlying graded space, which is greatly helpful in classifying related algebraic structures on a fixed vector space; see Example \ref{exam307} below, where we capitalize on the first order deformation of a specific 3-dimensional complex left-symmetric superalgebra to find all 3-dimensional complex simple left-symmetric superalgebras.

Determining whether two deformations are equivalent is one of core topics in any deformation theory. In particular,  given a left-symmetric color algebra $A$ and its standard bimodule $A$, the corresponding 2-coboundaries of degree 0 can be used to
measure the equivalence of two infinitesimal deformations (i.e., first-order deformations) of $A$; see Example \ref{exam3.9}.
As a direct consequence, we see that the second cohomology subspace $H^2_0(A,A)$ of degree 0  actually describes the equivalence classes of  infinitesimal deformations of $A$.

The second interesting topic in a deformation theory is to decide when an infinitesimal deformation can be extended nontrivially.
The following result demonstrates that the third cohomology subspace $H^3_0(A,A)$ of degree 0 can be understood as an obstruction to the extendability of infinitesimal deformations, generalizing a result due to Dzhumadil'daev in \cite[Corollary 4.2]{Dzh99}.

\begin{thm}\label{mt}
Let $A$ be finite-dimensional left-symmetric color algebra over a field $k$ and let $K$ be the field of fractions of the formal power series ring $k[[\uplambda]]$. Suppose that $A_K$ denotes the left-symmetric color algebra over $K$ induced by the tensor product $A\otimes_k K$. If $H^3_0(A_K,A_K)=0$, then every infinitesimal deformation is nontrivially $p$-extendable for all $p\geqslant 2$.
\end{thm}

More importantly, we will provide  a deeper understanding of the equivalence class of the trivial infinitesimal deformation of a left-symmetric color algebra $A$. In fact, deciding whether an infinitesimal deformation is equivalent to the trivial infinitesimal deformation allows us to introduce and study the notions of two families of linear operators on $A$: \textit{Nijenhuis operators} and  \textit{Rota-Baxter operators} of any weight, which also generalize the corresponding concepts in left-symmetric algebras and left-symmetric superalgebras. Recently, these two types of operators have gained significant popularity in the fields of nonassociative algebras and their applications in mathematical physics; see for example, \cite{PBG14, LSZB16,  TBGS23, RZ24}, and \cite{BGZ25}.

\vspace{2mm}
\noindent \textbf{Layout.}  This article is organized as follows. In Section \ref{sec2}, we recall some fundamental concepts and facts about left-symmetric color cohomology that was established in \cite{CZ25} originally.  In particular, we discuss 
extensions of the ground fields carefully, providing many basic constructions for left-symmetric color algebras and linear maps. 

As the core part of the article, Sections \ref{sec3} consists of three subsections. Firstly, we explore fundamental properties of 
deformations and infinitesimal deformations of a left-symmetric color algebra; see Proposition \ref{prop301} and Corollary \ref{coro305}.  Secondly, we study equivalent deformations and use them to interpret the second cohomology subspace of degree zero in Example \ref{exam3.9} and Remark \ref{rem310}. The third subsection is devoted to the extendability of infinitesimal deformations and a proof of Theorem \ref{mt}. 

In Section \ref{sec4}, we study Nijenhuis operators and  Rota-Baxter operators on a left-symmetric color algebra, generalizing some interesting  results in \cite{WSBL19}.  Moreover, we also explicitly compute the varieties of Nijenhuis operators and Rota-Baxter operators on a two-dimensional proper left-symmetric color algebra over the complex field $\C$.

\vspace{2mm}
\noindent \textbf{Conventions.} 
Throughout this article, we assume that $k$ denotes a field of any characteristic, $K$ a field extension of $k$ (of not necessary finite degree), and $G$ denotes an abelian group. We write $B_k(G)$ for the set of  all skew-symmetric bicharacters of $G$ over $k$; see Section \ref{sec2} below for the details. In Section \ref{sec3} and subsequent sections, we write $K$ for the field of fractions of the formal power series ring $k[[\uplambda]]$ in a new variable $\uplambda$. We use some standard notations: $k^\times$ denotes the subset of all nonzero elements in the field $k$, $\N^+$ denotes the set of all positive integers, $\N=\N^+\cup \{0\}$, and $\C$ denotes the complex field.

\section{Cohomology of Left-symmetric Color Algebras} \label{sec2}
\setcounter{equation}{0}
\renewcommand{\theequation}
{2.\arabic{equation}}
\setcounter{theorem}{0}
\renewcommand{\thetheorem}
{2.\arabic{theorem}}

\noindent Let us begin by recalling that a map $\upepsilon:G\times G\ra k^\times$
is called a \textit{bicharacter} of $G$ over $k$ if it is biadditive, i.e., $\upepsilon(a,b+c)=\upepsilon(a,b)\cdot \upepsilon(a,c)$ and $\upepsilon(a+b,c)=\upepsilon(a,c)\cdot \upepsilon(b,c)$ for all $a,b,c\in G$.  A bicharacter $\upepsilon$ of $G$ is said to be \textit{skew-symmetric} if  $\upepsilon(a,b)\cdot\upepsilon (b,a)=1$ for all $a,b\in G$. 

\subsection{Cohomology spaces}
This subsection recaps some fundamental concepts on left-symmetric color cohomology developed in \cite[Section 3]{CZ25}. A $G$-graded nonassociative algebra $A=\bigoplus_{a\in G}A_a$ is called a \textit{left-symmetric color algebra} over a field $k$ if there exists a skew-symmetric bicharacter $\upepsilon\in B_k(G)$ such that 
\begin{equation}
\label{lsa}
(xy)z-x(yz)=\upepsilon(|x|,|y|)\Big((yx)z-y(xz)\Big)
\end{equation}
for all homogeneous elements $x,y$, and $z\in A$.  For a finite-dimensional  $A$-bimodule $V$ over $k$,  the following $G$-graded space
\begin{equation}
\label{ }
C^n(A,V):=\Hom((\land^{n-1}_\upepsilon A)\otimes A,V)=\bigoplus_{c\in G}\Hom_c\left((\land^{n-1}_\upepsilon A)\otimes A,V\right)
\end{equation}
denotes  the space of the $n$-th cochains, where $\land^{n-1}_\upepsilon A$ denotes the $(n-1)$-th $\upepsilon$-exterior power of $A$.

The  \textit{$n$-th coboundary operator} $d_n:C^n(A,V)\ra C^{n+1}(A,V)$ is defined by
\begin{eqnarray*}
\label{ }
(d_nf)(x_1,\dots,x_{n+1})&:=&\sum_{i=1}^n(-1)^{i+1}\upepsilon\left(|f|+\sum_{j=1}^{i-1}|x_j|,|x_i|\right)x_i f(x_1,\dots,\ha{x_i},\dots,x_{n+1})\\
&&\hspace{-0.4cm}+\sum_{i=1}^n(-1)^{i+1}\upepsilon\left(|x_i|, \sum_{j=i+1}^n|x_j|\right)f(x_1,\dots,\ha{x_i},\dots,x_n,x_i)x_{n+1}\\
&&\hspace{-0.4cm}-\sum_{i=1}^n(-1)^{i+1}\upepsilon\left(|x_i|, \sum_{j=i+1}^n|x_j|\right)f(x_1,\dots,\ha{x_i},\dots,x_n,x_ix_{n+1})\\
&&\hspace{-0.4cm}+\sum_{1\leqslant j<i\leqslant n}(-1)^{i+1}\upepsilon\left(\sum_{s=j+1}^{i-1}|x_s|,|x_i|\right)f(x_1,\dots,[x_j,x_i],x_{j+1},\dots,\ha{x_i},\dots,x_{n+1})
\end{eqnarray*}
for all homogeneous elements $f\in C^n(A,V)$ and $x_1,\dots,x_{n+1}\in A$.  The space of \textit{$n$-th cocycles} is
\begin{equation}
\label{ }
Z^n(A,V)=\{f\in C^n(A,V)\mid d_n(f)=0\}
\end{equation}
and  the \textit{$n$-th coboundary} space $B^n(A,V)$ is defined as the image of $d_{n-1}$ in $C^{n}(A,V)$.  

For all $n\geqslant 1$, the key fact that $d_n\circ d_{n-1}=0$ allows us to define the quotient space
\begin{equation}
\label{sim}
H^n(A,V)=\frac{Z^{n}(A,V)}{B^{n}(A,V)}
\end{equation}
as the  \textit{$n$-th cohomology space} of $A$ with coefficients in $V$,  which is also a $G$-graded vector space with 
the $c$-component $Z^n_c(A,V)/B^{n}_c(A,V)$ for all $c\in G$; see \cite[Section 3.2]{CZ25} for more details.

\subsection{Extension of scalars} 

Note that each $\upepsilon\in B_k(G)$ can be regarded as a skew-symmetric bicharacter of $G$ over $K$ induced by the natural field injection from $k$ to $K$. Thus we may identify $B_k(G)$ with a subset of $B_K(G)$. We say that a skew-symmetric bicharacter $\upepsilon\in B_K(G)$ is \textit{defined over $k$} if $\upepsilon(a,b)\in k$ for all $a,b\in G$.

Let $V=\oplus_{g\in G}V_a$ be any $G$-graded vector space over $k$. Then the tensor product $V_K:=V\otimes_k K$, as a $K$-vector space, is also a $G$-graded vector space defined by
\begin{equation}
\label{ }
(V_K)_a:=V_a\otimes_k K
\end{equation}
for all $a\in G$. Note that $\dim_k(V)=\dim_K(V_K)$.

Suppose that $W$ denotes another  $G$-graded vector space over $k$. We may extend  a $k$-linear map $f:V\ra W$ to a $K$-linear map $f_K: V_K\ra W_K$ by
\begin{equation}
\label{ext}
f_K(v\otimes r):=f(v)\otimes r
\end{equation}
for all $v\in V$ and $r\in K$. Conversely, a $K$-linear map $F: V_K\ra W_K$ is \textit{defined over $k$} if there exists a $k$-linear map $f:V\ra W$ such that $F=f_K$. 

Moreover, a $k$-bilinear map $f:V\times V\ra W$ also can be extended to a $K$-bilinear map $f_K: V_K\times V_K\ra W_K$  by
\begin{equation}
\label{ext}
f_K(v\otimes r,w\otimes s):=f(v,w)\otimes rs
\end{equation}
for all $v,w\in V$ and $r,s\in K$. Conversely, a $K$-bilinear map $F: V_K\times V_K\ra W_K$
is \textit{defined over $k$} if there exists a $k$-bilinear map $f:V\times V\ra W$ such that
$F=f_K$.

The following result shows that a left-symmetric color algebra over $k$ can be extended as a 
left-symmetric color algebra over $K$.

\begin{prop} 
Let $(A,\mu)$ be a left-symmetric color algebra over $k$ with respect to $\upepsilon\in B_k(G)$. Then
$(A_K,\mu_K)$ is a left-symmetric color algebra over $K$ with respect to the same $\upepsilon$.
\end{prop}

\begin{proof}
Suppose that $x\otimes r, y\otimes s$, and $z\otimes t$ are arbitrary homogeneous elements in $A_K$ for $x,y,z\in A$ and $r,s,t\in K$. It follows from (\ref{ext}) that 
$$\mu_K(\mu_K(x\otimes r,y\otimes s),z\otimes t)=\mu_K(\mu(x,y)\otimes rs),z\otimes t)=\mu(\mu(x,y),z)\otimes rst.$$ Similarly, we see that $\mu_K(\mu_K(y\otimes s,x\otimes r),z\otimes t)=\mu(\mu(y,x),z)\otimes srt=\mu(\mu(y,x),z)\otimes rst$, $\mu_K(x\otimes r,\mu_K(y\otimes s,z\otimes t))=\mu(x,\mu(y,z))\otimes rst$, and $\mu_K(y\otimes s,\mu_K(x\otimes r,z\otimes t))=\mu(y,\mu(x,z))\otimes rst.$ Note that $|x\otimes r|=|x|, |y\otimes s|=|y|, |z\otimes t|=|z|.$  Hence,
\begin{equation} \label{ext2}
\begin{aligned}
&\mu_K(\mu_K(x\otimes r,y\otimes s),z\otimes t)-\mu_K(x\otimes r,\mu_K(y\otimes s,z\otimes t))-\\
&\upepsilon(|x\otimes r|,|y\otimes s|)\Big(\mu_K(\mu_K(y\otimes s,x\otimes r),z\otimes t)-\mu_K(y\otimes s,\mu_K(x\otimes r,z\otimes t))\Big)\\
=&~\Big(\mu(\mu(x,y),z)-\mu(x,\mu(y,z))-\upepsilon(|x|,|y|) \big(\mu(\mu(y,x),z)-\mu(y,\mu(x,z))\big)\Big)\otimes rst\\
=&~0\otimes rst=0
\end{aligned}
\end{equation}
because $(A,\mu)$ is a left-symmetric color algebra  with respect to $\upepsilon$. As $\mu_K: A_K\otimes_K A_K\ra A_K$ is linear, we see that the left-symmetric color identity (\ref{lsa}) follows for any homogeneous elements of $A_K$. Hence, $(A_K,\mu_K)$ is a left-symmetric color algebra  with respect to the same $\upepsilon$.
\end{proof}

We say that a left-symmetric color algebra $(A_K,\nu)$ with respect to $\upepsilon\in B_K(G)$ is \textit{defined over $k$}
if both $\upepsilon$ and $\nu$ are defined over $k$.

\begin{rem}{\rm
Suppose that $(A_K,\nu)$ is a left-symmetric color algebra defined over $k$ and we may assume $\nu=\mu_K$. Then
$(A,\mu)$ is also a left-symmetric color algebra over $k$ with respect to the same skew-symmetric bicharacter. 
To see that, setting $r=s=t=1$ in (\ref{ext2}) obtains
$$\Big(\mu(\mu(x,y),z)-\mu(x,\mu(y,z))-\upepsilon(|x|,|y|) \big(\mu(\mu(y,x),z)-\mu(y,\mu(x,z))\big)\Big)\otimes 1=0.$$
Since $A\cong A\otimes_kk$ as $k$-vector spaces,  the map $\uppi:A\ra A\otimes_kk$ defined by $v\mapsto v\otimes 1$ is injective. Thus $A_\textrm{tor}=\ker(\uppi)=\{0\}$. Note that for all $v\in A$ and $a\in k^\times$, $v\otimes a=0$ if and only if $v\in A_\textrm{tor}$. Applying this fact, we have
$$\mu(\mu(x,y),z)-\mu(x,\mu(y,z))-\upepsilon(|x|,|y|) \big(\mu(\mu(y,x),z)-\mu(y,\mu(x,z))\big)=0$$
which shows that $(A,\mu)$ is a left-symmetric color algebra over $k$.
\hbo}\end{rem}

\section{Deformations of Left-symmetric Color Algebras} \label{sec3}
\setcounter{equation}{0}
\renewcommand{\theequation}
{3.\arabic{equation}}
\setcounter{theorem}{0}
\renewcommand{\thetheorem}
{3.\arabic{theorem}}

\noindent This core section is to develop a deformation theory for left-symmetric color algebras. We focus on equivalent  
deformations and use the equivalence classes of infinitesimal deformations to interpret the second cohomology space of degree 0. In particular, we study the extendability of infinitesimal deformations and give a proof to Theorem \ref{mt}.

\subsection{One-parameter family of deformations}
Let $(A,\mu)$ be a finite-dimensional left-symmetric color algebra over $k$ with respect to $\upepsilon\in B_k(G)$ and $K$ be the field of fractions of the formal power series ring $k[[\uplambda]]$ in the variable $\uplambda$.

A generic element in the \textit{one-parameter family of deformations} of $A$ is a left-symmetric color algebraic structure $F_\uplambda$ on $A_K$ expressed by the following form
\begin{equation}
\label{opfd}
F_\uplambda(x,y)=f_0(x,y)+\uplambda\cdot f_1(x,y)+\uplambda^2\cdot f_2(x,y)+\cdots
\end{equation}
where $f_0(x,y):=\mu_K(x,y)$ and every $f_i(x,y):A_K\otimes_K A_K\ra A_K$  is a homogeneous linear map of degree 0 defined over $k$. Thus all $f_i\in C^2_0(A_K,A_K)$ are second cochains of degree 0 for $i\in\N^+$. 

We always assume that there exists some $m\in\N^+$ such that $f_i=0$ for all $i>m$, i.e.,  the number of nonzero 
terms of $F_\uplambda(x,y)$ is finite. 

\begin{prop} \label{prop301}
Let $(A_K,F_\uplambda)$ be a left-symmetric color algebra over $K$ with respect to $\upepsilon\in B_K(G)$. Then
for all $p\in\N$, we have
\begin{equation}
\label{lsci2}
\sum_{i,j\geqslant 0}^{i+j=p} \Big(f_i(f_j(x,y),z)-f_i(x,f_j(y,z))-\upepsilon(|x|,|y|)\big(f_i(f_j(y,x),z)-f_i(y,f_j(x,z))\big)\Big)=0
\end{equation}
for arbitrary homogeneous elements $x,y,z\in A_K$.
\end{prop}

\begin{proof}
By the left-symmetric color identity (\ref{lsa}), we see that
\begin{equation}
\label{lsci}
F_\uplambda(F_\uplambda(x,y),z)- F_\uplambda(x,F_\uplambda(y,z))=\upepsilon(|x|,|y|)\Big(F_\uplambda(F_\uplambda(y,x),z)-F_\uplambda(y,F_\uplambda(x,z))\Big).
\end{equation}
It follows from (\ref{opfd}) that 
\begin{eqnarray*}
F_\uplambda(F_\uplambda(x,y),z) & = & F_\uplambda\left(\sum_{j\geqslant 0} \uplambda^j\cdot f_j(x,y),z\right) = \sum_{j\geqslant 0} F_\uplambda\left(\uplambda^j\cdot f_j(x,y),z\right) \\
 &= & \sum_{j\geqslant 0} \sum_{i\geqslant 0} \uplambda^i\cdot f_i\Big(\uplambda^j\cdot f_j(x,y),z\Big)\\
 &=&\sum_{p\geqslant 0}\sum_{i,j\geqslant 0}^{i+j=p} \uplambda^p\cdot f_i\Big(f_j(x,y),z\Big)\\
 &=&\sum_{p\geqslant 0} \uplambda^p\left(\sum_{i,j\geqslant 0}^{i+j=p} f_i(f_j(x,y),z)\right).
\end{eqnarray*}
Similarly, 
\begin{eqnarray*}
F_\uplambda(x,F_\uplambda(y,z))&=& \sum_{p\geqslant 0} \uplambda^p\left(\sum_{i,j\geqslant 0}^{i+j=p} f_i(x,f_j(y,z))\right)\\
F_\uplambda(F_\uplambda(x,y),z)&=& \sum_{p\geqslant 0} \uplambda^p\left(\sum_{i,j\geqslant 0}^{i+j=p} f_i(f_j(y,x),z)\right)\\
F_\uplambda(y,F_\uplambda(x,z))&=& \sum_{p\geqslant 0} \uplambda^p\left(\sum_{i,j\geqslant 0}^{i+j=p} f_i(y,f_j(x,z))\right).
\end{eqnarray*}
Note that $\uplambda$ is an indeterminate. Substituting these equations together into 
(\ref{lsci}) and simplifying it, we see that (\ref{lsci2}) follows for all $p\in\N$.
\end{proof}

\begin{rem}{\rm
For $p\in\N$, we write $F_{\uplambda,p}(x,y)$ for the sum of the first $p+1$ terms of $F_{\uplambda}(x,y)$ and call it 
a \textit{deformation of $(A,\mu)$ of order $p$}. For instance, $F_{\uplambda,0}(x,y)=f_0(x,y)=\mu_K(x,y)$, and $F_{\uplambda,1}(x,y)=f_0(x,y)+\uplambda\cdot f_1(x,y)$, which is also called an \textit{infinitesimal deformation} of $(A,\mu)$.  
\hbo}\end{rem}

Writing $xy$ for $\mu_K(x,y)$ and taking $p=1$ in (\ref{lsci2})  imply the following result, which may give an interpretation of
2-cocycles of degree 0 in the left-symmetric color cohomology developed in \cite[Section 3.2]{CZ25}. 

\begin{coro} \label{coro3.3}
For all homogeneous elements $x,y,z\in A_K$, we have
$$f_1(xy,z)-f_1(x,yz)+f_1(x,y)z-xf_1(y,z)=\upepsilon(|x|,|y|)\Big(f_1(yx,z)-f_1(y,xz)+f_1(y,x)z-yf_1(x,z)\Big).$$
\end{coro}

\begin{rem}{\rm
This corollary shows that $f_1$ must be a 2-cocycle of degree 0 in $Z^2(A_K,A_K)$ with coefficients in the standard bimodule $A_K$; see \cite[Example 3.2]{CZ25}. 
\hbo}\end{rem}

More generally, for two 2-cochains $f$ and $g$ of degree 0, we define
\begin{equation}
\label{ }
(f\ast g)(x,y,z):=f(g(x,y),z)-f(x,g(y,z))-\upepsilon(|x|,|y|)\Big(f(g(y,x),z)-f(y,g(x,z))\Big)
\end{equation}
for homogeneous elements $x,y,z$.

\begin{coro} \label{coro305}
 If $F_{\uplambda,p}$ is a deformation of $A$ of order $p\geqslant 2$, then
\begin{equation}
\label{ }
\sum_{i=1}^{p-1} f_i\ast f_{p-i}=-d_2(f_p)
\end{equation}
where $d_2$ denotes the second coboundary operator in Section 2.1.
\end{coro}

\begin{rem}\label{rem3.6}
{\rm
Conversely, if $f_1,f_2,\dots,f_p$ (homogeneous linear maps from $A_K\otimes_K A_K$ to $A_K$ of degree zero defined over $k$)
satisfy (\ref{lsci2}), then $A_K$, together with $F_{\uplambda,p}(x,y)$, becomes a new left-symmetric color algebra over $K$. Clearly,
$F_{\uplambda,p}(x,y)$ is also defined over $k$ if $\uplambda$ is a scalar in $k$. Hence, restricting $A_K$ and $F_{\uplambda,p}(x,y)$ to $A$ may give us an approach to construct new left-symmetric color algebra structures on $A$. 
\hbo}\end{rem}

The following example demonstrates that each 3-dimensional complex simple left-symmetric superalgebras can be realized via a specific 3-dimensional complex left-symmetric superalgebra and its first order deformations.  

\begin{exam}\label{exam307}
{\rm
Let $A=A_0\oplus A_1$ be the 3-dimensional  left-symmetric superalgebra over $\C$ defined by
the following nonzero products:
$$xx=2x,~~ xy_1=y_1,~~ xy_2=y_2,~~ y_1y_2=x,~~ y_2y_1=-x$$
where $\{x\}$ spans $A_0$ and $\{y_1,y_2\}$ spans $A_1$. A direct computation shows that
$Z^2_0(A,A)$, the space of all 2-cocycles of degree 0, consists of all 2-cochains $f$ that have the following nonzero values:
\begin{equation}
\label{ }
f(x,x)=rx,~ f(x,y_1)=sy_2,~ f(x,y_2)=ty_1+ry_2
\end{equation}
where $r,s,t\in\C$. Hence, we may construct the following left-symmetric superalgebra structures on the underlying space of 
$A_K$ via using elements in $Z^2_0(A,A)$ and extensions to $A_K$:
\begin{eqnarray*}
A_\uplambda(t)&:& F_{\uplambda,1} (x,x)=(t+1)x,~  F_{\uplambda,1} (x,y_1)=y_1,~  F_{\uplambda,1} (x,y_2)=ty_2,~
F_{\uplambda,1} (y_1,y_2)=x,\\
&&  F_{\uplambda,1} (y_2,y_1)=-x,~\textrm{ where }t=e^{i\upalpha}\textrm{ and  }\upalpha\in[0,\uppi];\\
B_\uplambda&:&  F_{\uplambda,1} (x,x)=2x,~ F_{\uplambda,1} (x,y_1)=y_1,~
F_{\uplambda,1} (x,y_2)=y_1+y_2,~ F_{\uplambda,1} (y_1,y_2)=x,\\
&&F_{\uplambda,1} (y_2,y_1)=-x.
\end{eqnarray*}
The collection of the restrictions of the left-symmetric superalgebras $A_\uplambda(t)$ and $B_\uplambda$ to $A$ contains 
all the 3-dimensional complex simple left-symmetric superalgebras which actually have been classified by \cite[Sections 4 and 5]{ZB12}.
\hbo}\end{exam}

\subsection{Equivalent deformations}
We say that two left-symmetric color deformations $F_\uplambda$ and $E_\uplambda$ are \textit{equivalent}
if there exists a homogeneous linear map of the following form: 
$$P_\uplambda=p_0+\uplambda\cdot p_1+\uplambda^2\cdot p_2+\cdots$$
where $p_0$ denotes the identity map on $A_K$ and each $p_i\in C^1_0(A_K,A_K)=\Hom_0(A_K,A_K)$ is of degree 0 and defined over $k$, such that
\begin{equation}
\label{ed1}
F_\uplambda(P_\uplambda(x),P_\uplambda(y))=P_\uplambda(E_\uplambda(x,y))
\end{equation}
for all $x,y\in A_K$. Here we assume that $P_\uplambda$ has only finitely many nonzero terms.

\begin{prop}\label{prop3.8}
Let $F_\uplambda=\sum_{s=0}^\infty \uplambda^s\cdot f_s$ and $E_\uplambda=\sum_{j=0}^\infty \uplambda^j\cdot e_j$ be two 
left-symmetric color deformations. Then there exists a homogeneous linear map $P_\uplambda=\sum_{i=0}^\infty \uplambda^i\cdot p_i$  such that $F_\uplambda$ is equivalent to $E_\uplambda$ if and only if
\begin{equation}
\label{ed2}
\sum_{i,j,s\geqslant 0}^{i+j+s=p}f_s(p_i(x),p_j(y))=\sum_{i,j\geqslant 0}^{i+j=p}p_i(e_j(x,y))
\end{equation}
for all $x,y\in A_K$ and $p\in\N$.
\end{prop}

\begin{proof} 
Note that  $P_\uplambda(x)=\sum_{i=0}^\infty \uplambda^i\cdot p_i(x)$ and $P_\uplambda(y)=\sum_{j=0}^\infty \uplambda^j\cdot p_j(y)$. Thus
\begin{eqnarray*}
F_\uplambda(P_\uplambda(x),P_\uplambda(y))&=&\sum_{s=0}^\infty \uplambda^s\cdot f_s\left(\sum_{i=0}^\infty \uplambda^i\cdot p_i(x),\sum_{j=0}^\infty \uplambda^j\cdot p_j(y)\right)\\
&=& \sum_{p\geqslant 0}\uplambda^p\cdot \left(\sum_{i,j,s\geqslant 0}^{i+j+s=p}f_s(p_i(x),p_j(y))\right).
\end{eqnarray*}
Since $E_\uplambda(x,y)=\sum_{j\geqslant 0} \uplambda^j\cdot e_j(x,y)$, it follows that
\begin{eqnarray*}
P_\uplambda(E_\uplambda(x,y))&=&\sum_{i=0}^\infty \uplambda^i\cdot p_i\left(\sum_{j\geqslant 0} \uplambda^j\cdot e_j(x,y)\right)= \sum_{p\geqslant 0}\uplambda^p\cdot \left(\sum_{i,j\geqslant 0}^{i+j=p}p_i(e_j(x,y))\right).
\end{eqnarray*}
Therefore, the two equations (\ref{ed1}) and (\ref{ed2}) are equivalent because $\uplambda$ is an indeterminate. 
\end{proof}

\begin{exam}\label{exam3.9}
{\rm
Consider the case $p=1$ in (\ref{ed2}) and write $xy$ for the multiplication of two elements $x,y\in A_K$.
Note that $p_0$ is the identity map. Thus the left-hand side of (\ref{ed2})  becomes 
\begin{eqnarray*}
\sum_{i,j,s\geqslant 0}^{i+j+s=1}f_s(p_i(x),p_j(y))&=&f_1(p_0(x),p_0(y))+f_0(p_1(x),p_0(y))+f_0(p_0(x),p_1(y))\\
&=&f_1(x,y)+p_1(x)y+xp_1(y)
\end{eqnarray*}
where $f_1\in Z^2_0(A_K,A_K)$ and $p_1\in C^1_0(A_K,A_K)$. The right-hand side is 
$$\sum_{i,j\geqslant 0}^{i+j=1}p_i(e_j(x,y))=p_1(e_0(x,y))+p_0(e_1(x,y))=p_1(xy)+e_1(x,y).$$
Hence, the equation (\ref{ed2}) is equivalent to that $f_1(x,y)+p_1(x)y+xp_1(y)=p_1(xy)+e_1(x,y)$, i.e., 
\begin{equation}
\label{ed3}
e_1(x,y)-f_1(x,y)=p_1(x)y+xp_1(y)-p_1(xy)=d_1(p_1)(x,y)
\end{equation}
where $d_1$ denotes the first coboundary operator in Section 2.1, and $d_1(p_1)\in B^2_0(A_K,A_K)$; see \cite[Example 3.2]{CZ25}. 
This means that two infinitesimal deformations $F_\uplambda=f_0+\uplambda\cdot f_1$
and $E_\uplambda=e_0+\uplambda\cdot e_1$ are equivalent if and only if $f_1$ and $e_1$ are 
equal to each other, up to a 2-coboundary of degree 0. 
\hbo}\end{exam}

\begin{rem}\label{rem310}
{\rm
From this example,  we see that the second cohomology space $H^2_0(A,A)$  actually
describes the equivalence  classes of  infinitesimal deformations of $(A,\mu)$.
In particular, $F_\uplambda=E_\uplambda$ if and only if $p_1$ is an $\upepsilon$-derivation of degree 0; see 
\cite[Example 3.9]{CZ25}. Therefore, an $\upepsilon$-derivation $d$ of degree 0 can be used to construct an ``automorphism''  $P_\uplambda=p_0+\uplambda\cdot d$ of an infinitesimal deformation  $F_\uplambda=f_0+\uplambda\cdot f_1$.
\hbo}\end{rem}

\subsection{Extendable infinitesimal deformations}
Suppose that $F_\uplambda=f_0+\uplambda\cdot f_1$ denotes an infinitesimal deformation of a left-symmetric color algebra $(A,\mu)$ over $k$. For $p\geqslant 2$, we say that $F_\uplambda$ is \textit{nontrivially $p$-extendable} if there exist some nonzero $f_2,\dots,f_p\in C^2_0(A_K,A_K)$ defined over $k$ such that 
\begin{equation}
\label{eid}
\sum_{i=1}^{p-1} f_i\ast f_{p-i}=-d_2(f_p).
\end{equation}
This is equivalent to saying that $F_{\uplambda,p}:=f_0+\uplambda\cdot f_1+\uplambda^2\cdot f_2+\cdots+\uplambda^{p}\cdot f_p$
is a deformation of $A$ of order $p$.  We say that $F_\uplambda$ is \textit{trivially extended} to $F_{\uplambda,p}$  if 
$f_2=f_3=\cdots=f_p=0$.

To prove Theorem \ref{mt}, we need the following two lemmas. 

\begin{lem}\label{lem33-1}
Let $p\geqslant 2$.  If an infinitesimal deformation $F_\uplambda=f_0+\uplambda\cdot f_1$ can be extended to
$$F_{\uplambda,p-1}=\sum_{i=0}^{p-1} \uplambda^i\cdot f_i,$$
then
\begin{equation}
\label{ }
\sum_{i=1}^{p-1} f_i\ast f_{p-i}\in Z_0^3(A_K,A_K).
\end{equation}
\end{lem}

\begin{proof}
Clearly, $F_{\uplambda,p-1}$ can be trivially extended to a deformation of order $p$:
$$F_{\uplambda,p}=\sum_{i=0}^{p-1} \uplambda^i\cdot f_i+\uplambda^p\cdot f_p$$
where $f_p=0$. Thus $F_{\uplambda}$ can be extended to $F_{\uplambda,p}$.  By (\ref{eid}) we see that
$$d_3\left(\sum_{i=1}^{p-1} f_i\ast f_{p-i}\right)=d_3(-d_2(f_p))=-d_3(d_2(f_p))=0$$
which means that $\sum_{i=1}^{p-1} f_i\ast f_{p-i}$ belongs to the $0$-th component of $\ker(d_3)$, i.e., $Z^3_0(A_K,A_K)$. 
\end{proof}

\begin{lem}\label{lem33-2}
Every infinitesimal deformation is nontrivially $2$-extendable.
\end{lem}

\begin{proof}
We assume that $F_\uplambda=f_0+\uplambda\cdot f_1$ is an infinitesimal deformation. We first claim that $f_1\ast f_{1}=0$. In fact, let us write $xy$ for $f_0(x,y)$ and assume that $x,y,z$ are arbitrary homogeneous elements in $A_K$. Since $F_\uplambda=f_0+\uplambda\cdot f_1$ is a left-symmetric color algebra, we see that
\begin{eqnarray*}
F_\uplambda(F_\uplambda(x,y),z)-F_\uplambda(x,F_\uplambda(y,z))-\upepsilon(|x|,|y|)\Big(F_\uplambda(F_\uplambda(y,x),z)-F_\uplambda(y,F_\uplambda(x,z))\Big) &=&0
\end{eqnarray*}
in which
\begin{eqnarray*}
F_\uplambda(F_\uplambda(x,y),z) & = & (f_0+\uplambda\cdot f_1)((f_0+\uplambda\cdot f_1)(x,y),z) = (f_0+\uplambda\cdot f_1)(xy+\uplambda\cdot f_1(x,y),z) \\ 
 & = & (f_0+\uplambda\cdot f_1)(xy,z)+(f_0+\uplambda\cdot f_1)(\uplambda\cdot f_1(x,y),z) \\ 
  & = & (xy)z+\uplambda\cdot (f_1(xy,z)+ f_1(x,y)z)+\uplambda^2\cdot f_1(f_1(x,y),z) \\ 
F_\uplambda(x,F_\uplambda(y,z)) & = & x(yz)+\uplambda\cdot(f_1(x,yz)+xf_1(y,z))+\uplambda^2\cdot f_1(x,f_1(y,z)) \\ 
F_\uplambda(F_\uplambda(y,x),z)&=&(yx)z+\uplambda\cdot (f_1(yx,z)+ f_1(y,x)z)+\uplambda^2\cdot f_1(f_1(y,x),z) \\ 
F_\uplambda(y,F_\uplambda(x,z)) & = & y(xz)+\uplambda\cdot(f_1(y,xz)+yf_1(x,z))+\uplambda^2\cdot f_1(y,f_1(x,z)). 
\end{eqnarray*}
Note that $f_0$ gives a left-symmetric color algebra structure on $A_K$ and $f_1$ is a 2-cocycle of degree 0. Thus
it follows from (\ref{lsa}) and Corollary \ref{coro3.3} that
$$\uplambda^2\cdot\left(f_1(f_1(x,y),z)-f_1(x,f_1(y,z))-\upepsilon(|x|,|y|)\Big(f_1(f_1(y,x),z)-f_1(y,f_1(x,z))\Big)\right)=0$$
which implies that $(f_1\ast f_{1})(x,y,z)=0$ because $\uplambda$ is an indeterminate.  The claim holds. 

Now note that $Z^2_0(A_K,A_K)\neq 0$ because $f_1$ is nonzero element of $Z^2_0(A_K,A_K)$. We may take any nonzero $f_2\in Z^2_0(A_K,A_K)$. Then
$$-d_2(f_2)=0=f_1\ast f_{1}.$$
By Remark \ref{rem3.6}, we see that $F_\uplambda$ can be extended to $f_0+\uplambda\cdot f_1+\uplambda^2\cdot f_2.$ 
This extension is nontrivial because $f_2$ is not zero. 
\end{proof}

Now we are ready to prove Theorem \ref{mt}.

\begin{proof}[Proof of Theorem \ref{mt}]
We use induction on $p$ to show this result. Note that Lemma \ref{lem33-2} already proves this statement for the case of $p=2$. Let us consider the general case $p\geqslant 3$ and by induction hypothesis, we may assume that an infinitesimal deformation $F_\uplambda=f_0+\uplambda\cdot f_1$ can be extended nontrivially to
$F_{\uplambda,p-1}=\sum_{i=0}^{p-1} \uplambda^i\cdot f_i.$ Then
Lemma \ref{lem33-1}, together with the assumption that $H^3_0(A_K,A_K)=0$, implies that
$$\sum_{i=1}^{p-1} f_i\ast f_{p-i}\in Z_0^3(A_K,A_K)=B_0^3(A_K,A_K).$$
Hence, there exists an $f_2\in C^2_0(A_K,A_K)$ such that 
$$\sum_{i=1}^{p-1} f_i\ast f_{p-i}=-d_2(f_2)$$
and so by Remark \ref{rem3.6}, it follows that $F_\uplambda$ can be extended to
$F_{\uplambda,p}=\sum_{i=0}^{p} \uplambda^i\cdot f_i.$ As at least one element of $\{f_2,\dots,f_{p-1}\}$ is nonzero, there also exists 
at least one nonzero element in $\{f_2,\dots,f_{p-1},f_p\}$. In other words, $F_{\uplambda,p}$ is a nontrivial extension of $F_{\uplambda}$.
\end{proof}

\section{Operators on Left-symmetric Color Algebras} \label{sec4}
\setcounter{equation}{0}
\renewcommand{\theequation}
{4.\arabic{equation}}
\setcounter{theorem}{0}
\renewcommand{\thetheorem}
{4.\arabic{theorem}}

\noindent This section introduces and explores two families of operators on left-symmetric color algebras: Nijenhuis operators 
and Rota-Baxter operators, which naturally appear in the condition for an infinitesimal deformation to be trivial. We also study
the relationship between these two types of operators and describe their geometric structures in the two-dimensional case.

\subsection{Nijenhuis operators} 
Let $(A,\mu)$ be a left-symmetric color algebra over $k$ with respect to $\upepsilon\in B_k(G)$.
To understand the class of  infinitesimal deformations of $A$ that are equivalent to the trivial 
infinitesimal deformation $f_0+\uplambda\cdot 0$, we need to analyze  (\ref{ed2}) in Proposition \ref{prop3.8} for special degrees $p=1,2,$ and $3$. 

Let us begin with two general  infinitesimal deformations 
$F_\uplambda=f_0+\uplambda\cdot f_1$ and $E_\uplambda=f_0+\uplambda\cdot e_1$, that are equivalent via $P_\uplambda=p_0+\uplambda\cdot p_1$.  Note that $f_0=\mu_K$ and $p_0$ denotes the identity map. 
Write $xy$ for $\mu_K(x,y)$ for $x,y\in A_K$.
Setting $p=1$ in  (\ref{ed2}), we have seen in Example \ref{exam3.9} that
\begin{equation}\label{eq4.1}
e_1(x,y)=f_1(x,y)+p_1(x)y+xp_1(y)-p_1(xy)
\end{equation}
Setting $p=2$ and $3$ in  (\ref{ed2}), we obtain
\begin{eqnarray}
f_1(p_1(x),y)+f_1(x,p_1(y))+p_1(x)p_1(y) & = & p_1(e_1(x,y)) \label{eq4.2}\\
f_1(p_1(x),p_1(y)) & = & 0 
\end{eqnarray}
respectively. 

Now assume that $F_\uplambda$ is trivial (i.e., $f_1=0$). Then it follows from (\ref{eq4.1}) and
(\ref{eq4.2}) that
$e_1(x,y)=p_1(x)y+xp_1(y)-p_1(xy)$ and $p_1(x)p_1(y) = p_1(e_1(x,y))$. This means that $p_1$ must satisfy the following condition:
\begin{equation}
\label{ }
p_1(x)p_1(y) =p_1\Big(p_1(x)y+xp_1(y)-p_1(xy)\Big)
\end{equation}
for all $x,y\in A_K$.

Thus it is reasonable to introduce the following concept, which generalizes the notion of Nijenhuis operators on left-symmetric algebras; see for example, \cite[Section 4]{WSBL19}.

\begin{defn}{\rm
Let $A$ be a finite-dimensional left-symmetric color algebra over $k$ with respect to $\upepsilon\in B_k(G)$.
A homogeneous linear map $P:A\ra A$ is said to be a \textit{Nijenhuis operator} on $A$ if 
\begin{equation}
\label{NO1}
P(x)P(y) =\upepsilon(|P|+|x|,|P|)P\big(P(x)y\big)+P\big(xP(y)\big)-\upepsilon(|x|,|P|)P\big(P(xy)\big)
\end{equation}
for all homogeneous $x,y\in A$.
}\end{defn}

We write $\NN_A$ for the set of all Nijenhuis operators on a left-symmetric color algebra $A$.

\begin{prop}
Let  $P\in\NN_A$ be a Nijenhuis operator of $\upepsilon(|P|,|P|)=1$. Then 
\begin{equation}
\label{NO2}
P^i(x)P^j(y) =\upepsilon(i\cdot |P|+|x|,j\cdot|P|)P^j\big(P^i(x)y\big)+P^i\big(xP^j(y)\big)-\upepsilon(|x|,j\cdot|P|)P^{i+j}(xy)
\end{equation}
for all $i,j\in\N$ and homogeneous elements $x,y\in A$.
\end{prop}

\begin{proof} 
We may use induction on $i,j$ to prove this result.
The statement  holds immediately for the case either $i=0$ or $j=0$ because
$$xP^j(y) =\upepsilon(|x|,j\cdot|P|)P^j\big(xy\big)+xP^j(y)-\upepsilon(|x|,j\cdot|P|)P^{j}(xy)$$
and $P^i(x)y =P^i(x)y+P^i(xy)-P^{i}(xy)$. 

Now we assume that $i,j\in\N^+$. The induction hypothesis says that
\begin{eqnarray*}
P^{i-1}(x)P^{j}(y)&=&c_1\cdot P^{j}\big(P^{i-1}(x)y\big)+P^{i-1}\big(xP^{j}(y)\big)-c_2\cdot P^{i+j-1}(xy)\\
P^i(x)P^{j-1}(y) &=&c_3\cdot P^{j-1}\big(P^i(x)y\big)+P^i\big(xP^{j-1}(y)\big)-c_4\cdot P^{i+j-1}(xy)
\end{eqnarray*}
where
$c_1:=\upepsilon((i-1)\cdot|P|+|x|,j\cdot|P|), c_2:=\upepsilon(|x|,j\cdot|P|), c_3:=\upepsilon(i\cdot|P|+|x|,(j-1)\cdot|P|)$,
and $c_4:=\upepsilon(|x|,(j-1) \cdot|P|)$. 
Hence, it follows from (\ref{NO1}) that
\begin{eqnarray*}
P^{i}(x)P^{j}(y) & = & P\Big(P^{i-1}(x)\Big)P\Big(P^{j-1}(y)\Big) \\
 & = &b_1\cdot P\Big(P^{i}(x)P^{j-1}(y)\Big)+P\Big(P^{i-1}(x)P^{j}(y)\Big)-b_1\cdot P^2\big(P^{i-1}(x)P^{j-1}(y)\big)\\
 &=&b_1c_3\cdot P^{j}\big(P^i(x)y\big)+b_1\cdot P^{i+1}\big(xP^{j-1}(y)\big)-b_1c_4\cdot P^{i+j}(xy)\\
 &&+c_1\cdot P^{j+1}\big(P^{i-1}(x)y\big)+P^{i}\big(xP^{j}(y)\big)-c_2\cdot P^{i+j}(xy)\\
 &&-b_1b_2\cdot P^{j+1}\big(P^{i-1}(x)y\big)-b_1\cdot P^{i+1}\big(xP^{j-1}(y)\big)+b_1b_3\cdot P^{i+j}(xy)
\end{eqnarray*}
where $b_1:=\upepsilon((i-1)\cdot|P|+|x|,|P|), b_2:=\upepsilon((i-1)\cdot |P|+|x|,(j-1)\cdot|P|)$, and
$b_3:=\upepsilon(|x|,(j-1)\cdot|P|)$. Note that 
\begin{eqnarray*}
b_1b_2 & = & \upepsilon\Big((i-1)\cdot|P|+|x|,|P|\Big) \upepsilon\Big((i-1)\cdot |P|+|x|,(j-1)\cdot|P|\Big)\\
 & = & \upepsilon\Big((i-1)\cdot |P|+|x|,j\cdot|P|\Big)=c_1\\
b_1c_4 &=&\upepsilon\Big((i-1)\cdot|P|+|x|,|P|\Big)\upepsilon\Big(|x|,(j-1) \cdot|P|\Big)=b_1b_3.
\end{eqnarray*}
Hence, 
 \begin{equation}
\begin{aligned}
P^{i}(x)P^{j}(y) & = b_1c_3\cdot P^{j}\big(P^i(x)y\big)+P^{i}\big(xP^{j}(y)\big)-c_2\cdot P^{i+j}(xy)\\
&=b_1c_3\cdot P^j\big(P^i(x)y\big)+P^i\big(xP^j(y)\big)-\upepsilon(|x|,j\cdot|P|)\cdot P^{i+j}(xy)
\end{aligned}
 \end{equation}
Note that $\upepsilon(|P|,|P|)=1$ and
\begin{eqnarray*}
b_1c_3&=&\upepsilon\Big((i-1)\cdot|P|+|x|,|P|\Big)\upepsilon\Big(i\cdot|P|+|x|,(j-1)\cdot|P|\Big)\\
&=&\upepsilon(|P|,|P|)\upepsilon\Big((i-1)\cdot|P|+|x|,|P|\Big)\upepsilon\Big(i\cdot|P|+|x|,(j-1)\cdot|P|\Big)\\
&=&\upepsilon\Big(i\cdot|P|+|x|,|P|\Big)\upepsilon\Big(i\cdot|P|+|x|,(j-1)\cdot|P|\Big)\\
&=&\upepsilon\Big(i\cdot|P|+|x|,j\cdot|P|\Big).
\end{eqnarray*}
Therefore, $P^i(x)P^j(y) =\upepsilon(i\cdot |P|+|x|,j\cdot|P|)P^j\big(P^i(x)y\big)+P^i\big(xP^j(y)\big)-\upepsilon(|x|,j\cdot|P|)P^{i+j}(xy)$.
\end{proof}

\begin{coro}
Let  $P\in\NN_A$ be a Nijenhuis operator of $\upepsilon(|P|,|P|)=1$. Then 
$f(P)\in \NN_A$ for all $f\in k[x]_d$, where $k[x]_d$ denotes the space of power functions of degree $d$.  
\end{coro}

\begin{proof}
Setting $i=j=d$ in (\ref{NO2}) obtains this result. 
\end{proof}

\subsection{Rota-Baxter operators}

We extend the notion of Rota-Baxter operators on a left-symmetric algebra appeared in \cite[Section 5.1]{WSBL19} to the color version.

\begin{defn}{\rm
Let $A$ be a finite-dimensional left-symmetric color algebra over $k$ with respect to $\upepsilon\in B_k(G)$. Given a fixed 
$\uplambda\in k$, a homogeneous linear map $P:A\ra A$ is called a \textit{Rota-Baxter operator} of weight $\uplambda$ on $A$ if 
\begin{equation}
\label{RBO}
P(x)P(y) =\upepsilon(|P|+|x|,|P|)P\big(P(x)y\big)+P\big(xP(y)\big)+\uplambda \cdot P(xy)
\end{equation}
for all homogeneous elements $x,y\in A$. 
}\end{defn}

%In particular, Rota-Baxter operators of weight $0$ generalizes the concept of $\mathcal{O}$-operators of a left-symmetric algebra occurred  in \cite[Definition 2.9]{BLN10}. 

We also write $\B_A(\uplambda)$ for the set of all Rota-Baxter operators of weight $\uplambda$ on a left-symmetric color algebra $A$, and $\Hom_a(A,A)$ for the space of all homogeneous linear maps of degree $a$ from $A$ to itself. Note that $a\in G$. 

\begin{prop}
Let $P\in\NN_A$. Then $P\in\B_A(0)$ if and only if the restriction of $P^2$ to $A^2$ is zero.
\end{prop}

\begin{proof} 
Given two arbitrary homogeneous elements $x,y\in A$.
Suppose the restriction of $P^2$ to $A^2$ is zero, then $P^2(xy)=0$. It follows from (\ref{NO1}) that
$P(x)P(y) =\upepsilon(|P|+|x|,|P|)P\big(P(x)y\big)+P\big(xP(y)\big)$. Thus, $P\in\B_A(0)$. Conversely, together  $P\in\B_A(0)$ 
with $P\in\NN_A$ implies that $\upepsilon(|x|,|P|)P\big(P(xy)\big)=0$. Note that $\upepsilon(|x|,|P|)\neq 0$. Thus
$P\big(P(xy)\big)=0$. Namely, the restriction of $P^2$ to $A^2$ is zero.
\end{proof}

\begin{coro}
Let $P\in \Hom_a(A,A)$ and $P^2=0$. Then $P\in\NN_A$ if and only if $P\in\B_A(0)$.
\end{coro}

\begin{prop}
Let $P\in \Hom_a(A,A)$ and $P^2=P$. Then $P\in\NN_A$ if and only if $P\in\B_A(-1)$.
\end{prop}

\begin{proof}
Suppose $P\in\NN_A$ and $x,y\in A$ are homogeneous. Since $P^2=P$, it follows that $2a=a$ and so $a=|P|=0$. Thus
$P(x)P(y) -\upepsilon(|P|+|x|,|P|)P\big(P(x)y\big)-P\big(xP(y)\big)=-\upepsilon(|x|,|P|)P(xy)=-\upepsilon(|x|,0)P(xy)=-P(xy)$.
Hence, $P$ is a Rota-Baxter operator of weight $-1$. Conversely, if $P\in\B_A(-1)$, then it is immediate to verify 
(\ref{NO1}), which means that $P\in\NN_A$.
\end{proof}

\begin{prop}
Let $P\in \Hom_0(A,A)$ and $P^2=I_A$, the identity map on $A$. Then the following statements are equivalent:
\begin{enumerate}
  \item $P\in\PP_A$;
  \item $P+I_A\in\B_A(-2)$;
  \item $P-I_A\in\B_A(2)$.
\end{enumerate}
\end{prop}

\begin{proof} 
First of all, we note that $P\pm I_A$ both are homogeneous of degree zero because $P$ is of degree zero. Thus
$\upepsilon(|P\pm I_A|+|x|,|P\pm I_A|)=\upepsilon(|x|,|P\pm I_A|)=1$
for all homogeneous $x\in A$. 

To prove $(1) \RA (2)$, we need to show that 
\begin{equation}
\label{eq4rb}
(P+I_A)(x)(P+I_A)(y) -(P+I_A)\big((P+I_A)(x)y+x(P+I_A)(y)\big)=-2\cdot (P+I_A)(xy).
\end{equation}
In fact, for the left-hand side, it follows from the assumption that $P\in\NN_A$ that 
\begin{eqnarray*}
(P+I_A)(x)(P+I_A)(y) & = & P(x)P(y)+P(x)y+xP(y)+xy \\
 & = &P \big(P(x)y+xP(y)\big)-P^2(xy)+P(x)y+xP(y)+xy\\
 &=&P \big(P(x)y+xP(y)\big)+P(x)y+xP(y)\\
 &=&(P+I_A) \big(P(x)y+xP(y)\big).
\end{eqnarray*}
Moreover, 
\begin{eqnarray*}
(P+I_A)\big((P+I_A)(x)y+x(P+I_A)(y)\big) & = & (P+I_A)\big(P(x)y+xP(y)+2xy\big) \\
 & = & (P+I_A) \big(P(x)y+xP(y)\big)+(P+I_A)(2xy).
\end{eqnarray*}
Hence, (\ref{eq4rb}) follows.

To prove $(2) \RA (3)$, we may consider
\begin{eqnarray*}
(P-I_A)(x)(P-I_A)(y)& = & P(x)P(y)-P(x)y-xP(y)+xy \\
 & = & P(x)P(y)+P(x)y+xP(y)+xy -2(P(x)y+xP(y))\\ 
 &=& (P+I_A)(x)(P+I_A)(y) -2(P(x)y+xP(y))
\end{eqnarray*}
and \begin{eqnarray*}
&&(P-I_A)\big((P-I_A)(x)y+x(P-I_A)(y)\big)\\
 & = & (P-I_A)\big(P(x)y+xP(y)-2xy\big) \\
 & = & (P+I_A-2I_A)\big(P(x)y+xP(y)+2xy-4xy\big)\\
 &=&(P+I_A)\big(P(x)y+xP(y)+2xy\big)-(P+I_A)(4xy)-2\big(P(x)y+xP(y)+2xy\big)+8xy\\
 &=& (P+I_A)\big((P+I_A)(x)y+x(P+I_A)(y)\big)-2\big(P(x)y+xP(y)\big)-4\cdot P(xy).
\end{eqnarray*}
Hence, it follows from (\ref{eq4rb}) that 
\begin{eqnarray*}
&&(P-I_A)(x)(P-I_A)(y)- (P-I_A)\big((P-I_A)(x)y+x(P-I_A)(y)\big)\\
& = & 4\cdot P(xy)-2\cdot (P+I_A)(xy)=2\cdot (P-I_A)(xy).
\end{eqnarray*}
This shows that $P-I_A\in\B_A(2)$.

To prove $(3) \RA (1)$, we note that $(P-I_A)(x)(P-I_A)(y)= (P-I_A)\big((P-I_A)(x)y+x(P-I_A)(y)\big)+2\cdot (P-I_A)(xy)$, which implies that
\begin{eqnarray*}
P(x)P(y)-P(x)y-xP(y)+xy & = &  (P-I_A)\big(P(x)y+xP(y)-2xy\big)+2\cdot (P-I_A)(xy)\ \\
 & = & P\big(P(x)y+xP(y)\big)-(P(x)y+xP(y)).
\end{eqnarray*}
Thus, $P(x)P(y)=P\big(P(x)y+xP(y)\big)-xy=P\big(P(x)y+xP(y)\big)-P^2(xy)$, because $P^2=I_A$. Therefore, 
$P\in\NN_A$, as desired.
\end{proof}

\subsection{Computations for $2$-dimensional algebras}
It is well-known that explicit computations of low-dimensional cases are very important to the understanding of higher algebraic structures; see for example, \cite{Bai09, BM01, CLZ14, CZ17, CZZZ18}, and \cite{CRSZ26} for left-symmetric algebras and other nonassociative algebras. In this last subsection, we explicitly compute Nijenhuis operators and Rota-Baxter operators on 
2-dimensional left-symmetric color algebras over $\C$.

\begin{rem}{\rm
Note that 2-dimensional left-symmetric algebras and superalgebras over $\C$ have been classified already; see for example, 
\cite[Theorem 3.1]{ZB12} and \cite[Section 3]{ZB12}. The corresponding Nijenhuis geometry has been explored in \cite{Kon21}. Moreover, Rota-Baxter operators on 2-dimensional left-symmetric algebras and superalgebras over $\C$ were classified by \cite[Section 5]{LHB07} and \cite{AMM19} respectively. Thus, we will be working over the \textit{proper} color cases, i.e.,  left-symmetric color algebras we consider will neither be left-symmetric algebras nor left-symmetric superalgebras, although our method used here might be applied to the cases of 2-dimensional left-symmetric algebras and superalgebras.
\hbo}\end{rem}

Suppose that $G$ denotes a finite abelian group and $\upepsilon\in B_\C(G)$. We have already analyzed the variety of 2-dimensional left-symmetric color algebras with respect to $\upepsilon$ in \cite[Section 4.1]{CZ25}. There exists only one proper left-symmetric color algebras:
\begin{eqnarray*}
A_\upalpha&:&\C\cdot x\oplus \C\cdot y\textrm{ with the unique nonzero product }x^2=\upalpha\cdot y\textrm{ for }\upalpha\in \C^\times,\\
 &&2|x|=|y|,\textrm{ and }|x|\neq |y|\textrm{ both are nonzero in }G.
\end{eqnarray*}

We suppose that $P\in \NN_c(A_\upalpha)$ denotes a nonzero Nijenhuis operator on $A$ of degree $c\in G$, and $|x|=a$ and $|y|=b$. Note that $a\neq b$ and they both are not zero.
We may assume that $P(x)=rx+sy$ and $P(y)=tx+wy$ for some $r,s,t,w\in\C$.

Let us get started with the special case where $c=0$, i..e, $P$ is a degree-preserving map. Since $a\neq b$, it follows that
$P(x)=rx$ and $P(y)=wy$. Thus $P$ is a diagonal matrix of size $2$. On the other side,  to determine when a diagonal
$2$-by-$2$ matrix belongs to $\NN_0(A_\upalpha)$, we need to verify the  Nijenhuis condition (\ref{NO1}) for arbitrary two basis elements of $A_\upalpha$. According to the defining products in $A_\upalpha$, it suffices to consider the following equation:
$$P(x)P(x) =\upepsilon(|P|+|x|,|P|)P\big(P(x)x\big)+P\big(xP(x)\big)-\upepsilon(|x|,|P|)P\big(P(xx)\big)$$
which is equivalent to $r^2\cdot x^2=2r\cdot P(x^2)-P(P(x^2))$. Note that $x^2=\upalpha y$. Thus the fact that $P(y)=wy$ implies that
$$r^2\upalpha\cdot y=2rw\upalpha\cdot y-w^2\alpha\cdot y.$$
This means that $r^2-2rw+w^2=0$ as $\upalpha$ is nonzero.  Hence, $(r-w)^2=0$ and so $r=w$. In other words, $P$ must be 
a scalar matrix. This proves the following result.
 
\begin{prop}\label{prop409}
The subspace $\NN_0(A_\upalpha)$ consists of all scalar matrices of size $2$.
\end{prop}

Now we suppose $|P|=c\neq 0$.

\begin{lem}\label{lem410}
If $c\neq 0$, then $P(x)=sy$ and $P(y)=0$, for some $s\in\C$. 
\end{lem}

\begin{proof}
Note that $P(x)=rx+sy$ and the left-hand side is a homogeneous element of degree $|P|+|x|=c+a$. The right-hand side must be either 0 or a homogeneous element. As $a\neq b$, we see that $rx+sy$ is not homogeneous for all nonzero $r$ and $s$. Thus
$P(x)$ is equal to either $rx$ or $sy$. If $P(x)=rx$, then $c+a=a$ and so $c\neq 0$. This contracts with the assumption and it follows that $P(x)=sy$. A similar argument applies to $P(y)=tx+wy$, showing that $P(y)=tx$ for some $t\in \C$.

To see that $t=0$, we consider the Nijenhuis condition (\ref{NO1}) 
$$P(x)P(y) =\upepsilon(|P|+|x|,|P|)P\big(P(x)y\big)+P\big(xP(y)\big)-\upepsilon(|x|,|P|)P\big(P(xy)\big)$$
which implies that $P\big(xP(y)\big)=\upalpha t \cdot y=0$. As $\upalpha\in\C^\times$, we have $t=0$.
\end{proof}

\begin{prop}
The space $\NN(A_\upalpha)$ is isomorphic to the following matrix space 
$$\left\{\begin{pmatrix}
    r  &  0  \\
    s  &  r
\end{pmatrix}\mid r,s\in\C\right\}.$$
\end{prop}

\begin{proof}
It is immediate to verify that a homogeneous linear map $P$ determined by $P(x)=sy$ and $P(y)=0$ for some scalar $s$ is also a
Nijenhuis operator on $A_\upalpha$. This, together with Lemma \ref{lem410}, implies that $\PP_c(A_\upalpha)$ is isomorphic to the matrix space:
$$\left\{\begin{pmatrix}
    0  &  0  \\
    s  &  0
\end{pmatrix}\mid s\in\C\right\}.$$
Combining this with Proposition \ref{prop409} proves the statement.
\end{proof}

We may use the same procedure to describe Rota-Baxter operators of weight $\uplambda$ on $A_\upalpha$ but with a more complicated computations, because the value of $\uplambda$ plays a key role in the understanding of the equation:
$r^2-2wr-\uplambda w=0.$

For the convenience of readers, let us describe the main steps briefly. Suppose 
 $P\in \B_{A_\alpha}(\uplambda)$ denotes a nonzero Rota-Baxter operator of degree $c$, determined by $P(x)=rx+sy$ and $P(y)=tx+wy$ for some $r,s,t,w\in\C$.
 
\textsc{Case 1. } If $c=0$, then $P(x)=rx, P(y)=wy$, and so $P$ is diagonal. Note that 
$$P(x)P(x) =\upepsilon(|P|+|x|,|P|)P\big(P(x)x\big)+P\big(xP(x)\big)+\uplambda \cdot P(xx).$$
Thus $r^2\cdot x^2=2r\cdot P(x^2)+\uplambda \cdot P(x^2)$, which implies 
$$r^2-2rw-\uplambda w=0.$$
This means that $w=\frac{r^2}{2r+\uplambda}$. Hence, the subset of degree $0$ in $\B_{A_\upalpha}(\uplambda)$  can be identified with the following variety
$$\left\{\begin{pmatrix}
    r  &  0  \\
    0  &  \frac{r^2}{2r+\uplambda}
\end{pmatrix}\mid r\in\C\right\}.$$

\textsc{Case 2. } If $c\neq 0$, we also have $P(x)=sy$ and $P(y)=tx$ for some $s,t\in\C$. Moreover, the Rota-Baxter condition 
(\ref{RBO}) forces $t=0$, i.e., $P(y)=0$. A direct verification shows that any homogeneous linear map $P$ defined by 
$P(x)=sy$ and $P(y)=0$ is a Rota-Baxter operator of weight $\uplambda$. Combining this assertion with the conclusion obtained in the first case above, we eventually have

\begin{prop}
The set $\B_{A_\alpha}(\uplambda)$ can be identified with the following variety
$$\left\{\begin{pmatrix}
    r  &  0  \\
    s  &  \frac{r^2}{2r+\uplambda}
\end{pmatrix}\mid r,s\in\C\right\}.$$
\end{prop}

\begin{rem}{\rm
According to the referee’s suggestions, we conclude this article with several remarks on topics that may be of interest for future research in this area. First of all, we have seen in \cite[Theorem 1.1]{CZ25} that the cohomology groups of left-symmetric color algebras can be computed by the cohomology groups of Lie color algebras, demonstrating that there are close relationships between Lie color algebras and left-symmetric color algebras. This indicates that developing the theory of Rota-Baxter Lie algebras to the color version (for example, \cite{LHB07,TBGS23} and \cite{CRSZ25}) and exploring the potential connections between the theories of Rota-Baxter operators for Lie color algebras and left-symmetric algebras, might be interesting. Secondly,
it is well-known that Rota-Baxter operators of Lie algebras can be used to construct new left-symmetric algebras; see for example, \cite[Corollary 2.7]{AB08}. Thus, based on the results on left-symmetric colour algebras obtained in this article, extending this connection to the color version could be an interesting topic for further study. The last topic worth exploring might be how far a theory for left-symmetric color algebras is from the theory for left-symmetric superalgebras. A special example is to consider the all left-symmetric color (or super) algebra structures on a Lie color (or super) algebra (see \cite{DZ23} and \cite[Remark 3.15]{CZ25}), which has inspired many researches recently; see \cite{CW24} and \cite{BBE25}.  Hence, exploring the difference between the cohomology (and deformations) of left-symmetric color algebras and those of left-symmetric superalgebras could be an interesting direction.
\hbo}\end{rem}

\vspace{2mm}
\noindent \textbf{Acknowledgements}.  This research was partially supported by NNSF of China (Grant No. 12561003) and USask (No. APEF-121159). The author would like to thank the anonymous referee and the editor for their careful reading, constructive comments, and suggestions.

%%%%%%%%%%%%%%%%%%%%%%%%%%References%%%%%%%%%%%%%%%%%%%%%%%%

\raggedright
\end{document}